\documentclass[10pt]{amsart}

\usepackage{amsaddr}
\usepackage{bbm}
\usepackage[normalem]{ulem}
\usepackage{mathrsfs}
\usepackage[usenames,dvipsnames,svgnames,table]{xcolor}
\usepackage{graphicx}
\usepackage{tikz-cd}
\usepackage{comment}
\usepackage[pagebackref,colorlinks=true,linkcolor=BrickRed,citecolor=OliveGreen,pdfstartview=FitH]{hyperref}

\title{The spectrum of a solenoid}

\author{Raymond Lei}
\address{Department of Mathematics\\ Indiana University Bloomington\\ 
831 East 3rd St. \\
Bloomington, IN 47405 \\ \tt{leim@iu.edu}}

\makeatletter
\@ifpackageloaded{MnSymbol}{}{%
  \usepackage{pifont}

}
\makeatother

\makeatletter
\newtheorem*{rep@theorem}{\rep@title}
\newcommand{\newreptheorem}[2]{%
\newenvironment{rep#1}[1]{%
 \def\rep@title{#2 \ref{##1}}%
 \begin{rep@theorem}}%
 {\end{rep@theorem}}}
\makeatother

\newtheorem{innercustomthm}{Corollary}
\newenvironment{customthm}[1]
  {\renewcommand\theinnercustomthm{#1}\innercustomthm}
  {\endinnercustomthm}

\newreptheorem{theorem}{Theorem}

\newreptheorem{lemma}{Lemma}
\newtheorem{thm}{Theorem}[section]
\newtheorem{ass}[thm]{Assumption}
\newtheorem{coro}[thm]{Corollary}
\newtheorem{lem}[thm]{Lemma}
\newtheorem{prop}[thm]{Proposition}
\newtheorem{remark}[thm]{Remark}

\newtheorem*{thm*}{Theorem}
\newtheorem*{coro*}{Corollary}

\theoremstyle{definition}

\newtheorem{defn}[thm]{Definition}
\newtheorem{eg}[thm]{Example}

\theoremstyle{remark}

\newcommand{\Pcal}{{\mathcal P}} 
\newcommand{\Gcal}{{\mathcal G}} 
\newcommand{\Ecal}{{\mathcal E}}  
\newcommand{\Hcal}{{\mathcal H}}
\newcommand{\Dcal}{{\mathcal D}}

\newcommand{\Mcal}{{\mathcal M}}
\newcommand{\Bcal}{{\mathcal B}}
\newcommand{\Fcal}{{\mathcal F}}
\newcommand{\Ucal}{{\mathcal U}}
\newcommand{\Ccal}{{\mathcal C}}

\newcommand{\Zbb}{{\mathbb Z}}
\newcommand{\Rbb}{{\mathbb R}}
\newcommand{\Cbb}{{\mathbb C}}
\newcommand{\Nbb}{{\mathbb N}}
\newcommand{\Hbb}{{\mathbb H}}

\begin{document}

\maketitle

\begin{abstract}
    Given a sequence of regular finite coverings of complete Riemannian manifolds, we consider the covering solenoid associated with the sequence. We study the leaf-wise Laplacian on the covering solenoid. The main result is that the spectrum of the Laplacian on the covering solenoid equals the closure of the union of the spectra of the manifolds in the sequence. We offer an equivalent statement of Selberg's $1/4$ conjecture.
\end{abstract}

\section*{Introduction}

Consider a sequence of regular finite coverings of complete 
Riemannian manifolds
$X_1\leftarrow X_2\leftarrow X_3\leftarrow\cdots$. We will call the inverse 
limit of this sequence the {\em covering solenoid} of the sequence and will 
denote it by $X_{\infty}$. The covering solenoid has a natural foliation 
whose leaves are Riemannian manifolds. This
allows us to define a self-adjoint leaf-wise Laplacian operator
$\Delta_{X_{\infty}}$ on $X_{\infty}$, see Section 3. 

\begin{reptheorem}{mainTheorem}
The spectrum of $\Delta_{X_{\infty}}$ equals the closure of the union of all Laplace spectra of $X_i$, $i\in\Nbb$.
\end{reptheorem}

Sequences of coverings arise, more generally, from certain inverse systems of finite covers. For instance, consider a connected manifold $X$ such that $\pi_1(X)$ is finitely generated and residually finite. The collection of finite coverings of $X$ is an inverse system. The inverse limit of the system is homeomorphic to the covering solenoid associated with the sequence $X_1\leftarrow X_2\leftarrow X_3\leftarrow\cdots$ where $G_i$ is the intersection of the subgroups of $\pi_1(X)$ of index smaller than or equal to $i$, $X_i$ is $G_i\backslash\tilde{X}$, and $\tilde{X}$ is the universal cover of $X_1$ for each integer $i$. 

Another instance concerns the collection of all congruence covers $X(n):=\Gamma(n)\backslash \Hbb$ of the modular surface $SL_2(\Zbb) \backslash \Hbb$. Let $X(\infty)$ denote the inverse limit of 
$X(\ell(2)) \leftarrow X(\ell(3))\leftarrow X(\ell(4)) \leftarrow \cdots$ where $\ell(k)$ is the least common multiple of $2,\ldots,k$ for each integer $k$. As a corollary of Theorem 4.7,
we can restate Selberg's $1/4$ conjecture concerning the first nonzero eigenvalue, $\lambda_1$, of all congruence covers \cite{Selberg65}
\cite{Bergeron16}:

\begin{customthm}{5.1}\label{eight}
Selberg's 1/4 conjecture is true if and only if the spectrum of $\Delta_{X(\infty)}$ does not intersect $(0,1/4)$.
\end{customthm}

In general, a covering solenoid is an object of some interest.
See, for example,  \cite{McCord65},
\cite{Sullivan93}, \cite{Saric09}, \cite{Clark-Hurder11}. 
The covering solenoid is connected to the study of dynamic systems: 
The 2-adic solenoid arises as a one-dimensional expanding attractor, 
or Smale-Williams attractor, and forms an important example in the 
theory of hyperbolic dynamical systems \cite{Pesin97},\cite{Sullivan74}.

The article is organized as follows.

In section 1, we give the 
definition of $X_{\infty}$ and discuss some special cases.

In section 2, we define a measure on $X_{\infty}$. This measure is natural 
in terms of the following: In the case where each $X_i$ is a locally compact abelian group,
$X_{\infty}$ is a locally compact abelian group and the measure defined on $X_{\infty}$
is a Haar measure. With this measure, we study the $L^2$ space 
of $X_{\infty}$.

In section 3, we make use of the fact that $X_{\infty}$ is foliated and define 
a leaf-wise Laplacian. For the case where each $X_i$ is complete, the Laplacian is essentially self-adjoint. We show that the leaf-wise Laplacian on an appropriate domain is also essentially self-adjoint in this case. 

In section 4, we show that $L^2(X_{\infty})$ 
is spanned by pullbacks of Laplacian eigenfunctions on all $X_i$, where each $X_i$ is a closed manifold (Proposition 4.1). This allows us to define a resolution of the identity $E$, of the self-adjoint Laplacian on $X_{\infty}$, with the pullbacks of these eigenfunctions. For a more general case where each $X_i$ is a complete manifold. Let $E^i$ denote the resolution of the identity for Laplacian on $X_i$. We then similarly form a resolution of the identity $E$ of the self-adjoint Laplacian on $X_{\infty}$ as follows: For each Lebesgue measurable set $\omega\subset\Rbb$, $E(\omega)$ is the linear combination of pullbacks of $E^i(\omega)$ acting on corresponding $L^2(X_i)$ subspaces, $i\in\Nbb$ (Theorem 4.6). This leads to the main result Theorem 4.7.

In section 5, we give details explaining how Corollary 5.1
follows from Theorem \ref{mainTheorem}.


\section{Covering Solenoids}

Consider a sequence of regular finite coverings of connected, locally path connected, and semi-locally simply connected spaces 
 $X_1\leftarrow X_2\leftarrow X_3\leftarrow\cdots$. 
For every $i,j\in\Nbb$ with $j>i$,
denote the covering map from $X_j$ to $X_i$ by $P_{j,i}$. 
A sequence of regular finite coverings of connected manifolds is an example of such sequences.

\begin{defn}
The {\em covering solenoid associated with the sequence $X_1\leftarrow X_2\leftarrow X_3\leftarrow\cdots$} is 
defined to be
\[ X_{\infty}~ 
:=~ 
\left\{ 
\left. 
(x_1,x_2,x_3,\ldots)\in\prod X_i\, 
\right|
\,\,
P_{j,i}(x_j)=x_i,\,\,\textrm{for every}\,\,i,j\in\Nbb, j>i \right\}.
\]
\end{defn}

In other words, the covering solenoid of is the inverse limit
of the inverse system  $X_1\leftarrow X_2\leftarrow X_3\leftarrow\cdots$.

The infinite product $\prod X_i$ is equipped with the product topology. 
$X_{\infty}$ is equipped with the subspace topology. 
$X_{\infty}$ is a closed subset of $\prod X_i$. 

Let $P_{\infty,i}$ denote the projection map from $X_{\infty}$ to $X_i$ 
for each $i$. Let $G_i$ denote the fundamental group of $X_i$.
Since the finite covering maps in the sequence are regular coverings, 
$G_1/G_i$ is a finite group for each integer $i$. Let $G_{\infty}$ 
denote the 
inverse limit of the sequence of finite groups $G_1/G_i$, $i\in\Nbb$. 
$G_{\infty}$ is a group.

For each $i$, $G_1/G_i$ is the covering transformation group of $P_{i,1}:X_i\to
X_1$. 
For each $i$, $G_1/G_i$ acts on $X_i$. As a result, $G_{\infty}$ acts on
$X_{\infty}$.

\begin{thm}\cite[Theorem 5.6]{McCord65}
If each $X_i$ is path connected and semi-locally simply connected, then
$P_{\infty,1}: X_{\infty} \to X_1$ is a principal $G_{\infty}$-bundle.
\end{thm}

Notice that $G_{\infty}$ preserves each fiber of $P_{\infty,1}$. 

\begin{remark} \label{remk:two-descriptions}
For each integer $i$, there exists a natural bijection $h$ from
the inverse limit of the sequence 
$X_i\leftarrow X_{i+1}\leftarrow X_{i+2}\leftarrow\cdots$ 
to the inverse limit of 
$X_1\leftarrow X_2\leftarrow X_3\leftarrow\cdots$. 
To be specific, the map $h$ is defined
as follows:
\[(x_i,x_{i+1},x_{i+2},\ldots)\mapsto
(P_{i,1}(x_i),P_{i,2}(x_i),\ldots,P_{i,i-1}(x_i),x_i,x_{i+1},x_{i+2},\ldots).\]

As one can check, the bijection $h$ is also a homeomorphism. 
So $\varprojlim_{j\geq i}X_j$ is homeomorphic to 
$\varprojlim_{j\geq 1}X_j=X_{\infty}$.

Theorem 1.2 implies that $\varprojlim_{j\geq i}X_j\to X_i$ is a 
principal fiber bundle. The group $\varprojlim_{j\geq i}G_i/G_j$ acts on $\varprojlim_{j\geq i}X_j$ and preserves each fiber of $\varprojlim_{j\geq i}X_j\to X_i$. Notice that,under $h$, the fiber over each $x \in X_1$ for the bundle
$P_{\infty,1}: X_{\infty} \to X_1$
is the image of the 
finite union of the fibers over points  $ y \in P_{i,1}^{-1}(x)$
for the bundle $P_{\infty,i}: \varprojlim_{j\geq i}X_j \to X_i$. To be specific,

\begin{equation} \label{eqn:fiber-decomposition}
P_{\infty,1}^{-1}(x)~
=~
h \left( \bigcup_{y \in P_{i,1}^{-1}(x)}P_{\infty,i}^{-1}(y)\right).
\end{equation}

Note that the projection map $P_{\infty,i}$ from
$X_{\infty}$ to $X_i$ differs from the projection map from
$\varprojlim_{j\geq i}X_j$ to $X_i$
by the homeomorphism $h$.
Thus we
can and will view 
$P_{\infty,i}:X_{\infty}\to X_i$ as a principal $\varprojlim_{j\geq i}G_i/G_j$-bundle.%
\end{remark}

\begin{eg}
Consider the inverse system of all the finite covers of a manifold $X$ whose fundamental group is finitely generated and residually finite. Note that a group being residually finite implies that the intersection of all the finite index subgroups is trivial. 

The collection of all finite covers of $X$ forms a partially ordered set (inverse system). The partial ordering
$\prec$ is defined as follows: For two arbitrary finite covers $X'$ and $X''$, $X'\prec X''$ if 
$X''$ is a finite cover of $X'$. The partially ordered set is directed because for arbitrary finite covers $X'$ and $X''$, there is a finite cover of $X$ that covers both $X'$ and $X''$.

There exists a sequence of regular finite covers of $X$ such that the associated covering solenoid is homeomorphic to the inverse limit of the system, as follows:

Let $G$ denote the fundamental group of $X$. For each positive integer $i$, let $G_i$ be the
intersection of all the finite index subgroups of $G$
with an index smaller than or equal to $i$. Since
$G$ is finitely generated, for every
$i\in\mathbb{N}$, there are finitely many subgroups of $G$ that are of index $i$ or less. 
Thus $G_i$ is a finite index subgroup of $G$. In particular, $G_1=G$. Since we assumed that $G$ is residually finite,
$\bigcap G_i=id$. Since $G_{i+1}\leq G_{i}$ for every $i$ and $\bigcap G_i=id$, for every finite index subgroup 
$H$ of $G$, there exists an integer $i$ such that $G_i\leq H$. Let $\tilde{X}$ be the universal cover of $X_1$. 
For each finite cover $X'$ of $X$, there exists an integer $i$ such that $G_i$ is a subgroup of the fundamental group of $X'$. So there exists $X_i=\tilde{X}/G_i$
as a cover of $X'$. Therefore the collection of all $X_i$ forms a cofinal subsystem of the
inverse system of all the finite covers of $X$. As a result, the inverse
limit of all $X_i$ is
homeomorphic
to the inverse limit of all the finite covers of $X$
\cite[Lemma 1.1.9]{Ribes-Zalesskii10}.
\end{eg}

\section{Measure}

In this section, consider a sequence of 
regular finite coverings of connected,
locally path connected, semi-locally simply connected, second countable,
locally compact, and Hausdorff spaces $X_1\leftarrow X_2\leftarrow
X_3\leftarrow\cdots$. The last three
conditions are necessary for the
constructions in this section. Notice that, in this case, $X_{\infty}$ is a locally
compact Hausdorff space \cite[Theorem 5]{Stone79}. Again, 
a sequence of regular finite coverings of
connected manifolds is an example of such sequences.

In this section, we define a measure on the covering solenoid 
$X_{\infty}$ and discuss the $L^2$ space on the covering solenoid with respect to this measure.
To define the measure we use the  principal $G_{\infty}$-bundle
structure of $X_{\infty}$.

We first introduce the notion of Baire sets. The {\em Baire $\sigma$-algebra} of a locally compact Hausdorff topological space is the smallest $\sigma$-algebra containing all compact sets that are countable intersections of open sets. A member of a Baire $\sigma$-algebra is called a
{\em Baire set}.
A {\em Baire measure} is a measure defined on 
the Baire $\sigma$-algebra. For example, a Borel measure on a manifold and a Haar measure on a locally compact second countable topological group are both Baire measures.

In general, a Baire measure on the total space of a principal fiber bundle may be
induced from Baire measures on the base space and the fiber 
so that the measure is locally a product measure \cite{Goetz59}. 
To be more precise, let  $\pi:\mathcal{E}\to\mathcal{B}$ 
be a principal fiber bundle with fiber
$\mathcal{F}$ and group $\mathcal{G}$ where 
$\mathcal{E}$, $\mathcal{B}$ and $\mathcal{F}$ are locally compact. 
Given a $\mathcal{G}$-invariant Baire measure $\nu$ on $\mathcal{F}$, 
define a measure $\nu_b$ on each fiber $\pi^{-1}(b)$
by pushing forward the measure $\nu$ with a local trivialization 
map. Since $\nu$ is $\Gcal$-invariant, the measure $\nu_b$ does 
not depend on the choice of trivialization.

Given a Baire measure $\mu_{\mathcal{B}}$ on $\mathcal{B}$, 
define for each Baire subset $Z$ of
$\mathcal{E}$, 
\begin{equation}\label{product-measure}
\mu_{\mathcal{E}}(Z)
~=~
\int_{\mathcal{B}}\nu_b(Z\cap\pi^{-1}(b))\, d\mu_{\mathcal{B}}(b).
\end{equation}

It follows from Fubini's theorem that $\mu_{\Ecal}$ is a Baire measure.
Moreover, for each local trivialization $\psi$, we have
$\psi_*(\mu_{\Ecal})= \nu \times \mu_{\Bcal}$  \cite[Theorem 1]{Goetz59}.

\begin{defn}
Given a $\sigma$-finite Baire measure $\mu_{\Bcal}$ and a $\mathcal{G}$-invariant Baire measure $\nu$ of finite volume,
a measure $\mu_{\Ecal}$ that satisfies $\psi_*(\mu_{\Ecal})= \nu \times \mu_{\Bcal}$, for each local trivialization $\psi$,
is called the {\em product} of $\mu_{\Bcal}$ and $\nu$.
\end{defn}

\begin{prop}
If $\Bcal$ is second countable, locally compact, and Hausdorff, the product of $\mu_{\Bcal}$ and $\nu$ 
is unique.
\end{prop}

\begin{proof}
Suppose $\mu_{+}$ and $\mu_{-}$ are both products of
$\mu_{\Bcal}$ and $\nu$.

Since $\Bcal$ is second countable, there exists a countable
collection $\Ucal$ of open subsets of $\Bcal$ such that any 
open subset of $\Bcal$ can be written
as a union of some subfamily of $\Ucal$. 
In particular, $\Bcal$ can be written
as a union of countably many open sets $\{U_i\}$. 
And $\{U_i\}$ can be chosen
such that there exists a local trivialization on each $U_i$. 
Since $\mu_{+}$ and $\mu_{-}$ are both $\sigma$-finite,
$\{U_i\}$ can be chosen such that each
$U_i$ is of finite $\mu_{\pm}$ measure. Let $W_i$ denote the set
$U_i-(\cup_{j>i}U_j)$. Since $\Bcal$ is locally compact Hausdorff,
Baire $\sigma$-algebra equals Borel $\sigma$-algebra
\cite[p.216]{Folland99}. So each open set $U_i$ is a Baire set.
Furthermore, each $W_i$ is a Baire set. So each $W_i$ 
is $\mu_{\pm}$ measurable and of finite measure.
Notice that $W_i$ is a subset of $U_i$, the restriction of the 
local trivialization 
is still a homeomorphism. So there exists a local trivialization 
on each $W_i$.

Given a Baire set $Z \subset \Fcal$,
we have $\mu_{\pm}(Z) = \sum \mu_{\pm}( \pi^{-1}(W_i) \cap Z)$.
Since $\mu_{+}$ and $\mu_{-}$ are both product of $\mu_{\Bcal}$ and
$\nu$, we have $\mu_{+}(\pi^{-1}(W_i) \cap Z)=\mu_{-}( \pi^{-1}(W_i) \cap Z)$ 
for each $i$. Therefore $\mu_+(Z) = \mu_-(Z)$.
\end{proof}

The definition of $\mu_{\mathcal{E}}$ implies that if $\nu$ is of unit volume,
then $\mu(\pi^{-1}(Y))=\mu_{\mathcal{B}}(Y)$ for every Baire set $Y\subset
\mathcal{B}$.

Notice that a Borel measure on $X_1$ and a Haar
measure on the locally compact second countable topological group $G_{\infty}$ are
both Baire measures.

We apply formula (\ref{product-measure}) to the $G_{\infty}$ principal bundle 
$P_{\infty,1}:X_{\infty}\to X_1$. 
Given a Borel measure $\mu_1$ on $X_1$ and a unit volume Haar measure $\nu$ on $G_{\infty}$, 
there exists a Baire measure on $X_{\infty}$ as follows:

\begin{defn}
For each $x_1\in X_1$, let $\nu_{x_1}$ denote the 
pushforward measure on the fiber $P_{\infty,1}^{-1}(x_1)$. 
The Baire measure $\mu_{\infty}$ associated to the measures $\nu$ and $\mu_1$ is defined 
by $\mu_{\infty}(Z)=\int_{X_{1}} \nu_{x_{1}} (Z\cap P_{\infty,1}^{-1}(x_1))d\mu_1(x_1)$ for each Baire set $Z\subset X_{\infty}$. 
\end{defn}

For each $i$, 
we view the finite covering $P_{i,1}: X_i \to X_1$
as a principal $G_1/G_i$-bundle.
We also apply Definition 2.1 
to this principal bundle to obtain a measure $\mu_i$ on $X_i$. 
In particular, let $\nu_{i}$ be the unit volume Haar measure on the fiber $G_1/G_i$.
There exists a Baire measure $\mu_i$ on $X_i$ that satisfies
$\mu_i(E)=\int_{X_i}\nu_{i,x_1}(E\cap P_{i,1}^{-1}(x_1))d\mu_1(x_1)$
for each Baire set $E\subset X_i$ where $\nu_{i,x_1}$ is the 
pushforward measure on the fiber $P_{i,1}^{-1}(x_1)$ by the 
local trivialization map. 

For each $i$, by viewing $X_{\infty}$ 
as a principal bundle over the base $X_i$, as in Remark 1.3,
we obtain a measure on $X_{\infty}$ associated 
to $\mu_i$ on $X_i$ and the unit volume Haar measure $\nu_i$ on $\varprojlim_{j\geq i}G_i/G_j$.

\begin{lem}
The measure on $X_{\infty}$ induced by $\mu_1$ and $\nu$ equals the measure induced by the measure $\mu_i$ and $\nu_i$ for each $i$.
\end{lem}

\begin{proof}
Recall from Remark \ref{remk:two-descriptions} the natural 
homeomorphism $h$ from  $\varprojlim_{j\geq i}X_j$ to $X_{\infty}$.
And recall that the fiber over $x \in X_1$ for the bundle
$P_{\infty,1}: X_{\infty} \to X_1$
is the image under $h$ of the 
finite union of the fibers over points  $y\in P_{i,1}^{-1}(x)$
for the bundle $P_{\infty,k}: \varprojlim_{j\geq i}X_j \to X_k$.
See equation (\ref{eqn:fiber-decomposition}). It follows that
for each Baire set $Z \subset \varprojlim_{j\geq i}X_j$ we have 
\[
\int_{X_{i}} \nu_{i,y} \left(Z\cap P_{\infty,i}^{-1}(y) \right)\,
d\mu_i(y)~
=~
\int_{X_{1}} \nu_{x} \left(h(Z)\cap P_{\infty,1}^{-1}(x) \right)\,
d\mu_1(x).
\]

Since $h$ is a homeomorphism, $h$ is an isomorphism of Baire $\sigma$-algebras. The claim follows.
\end{proof}

\begin{eg}
Suppose that $X_1\leftarrow X_2\leftarrow\cdots$ is a sequence 
of locally compact abelian groups and each covering map 
$P_{j,i}: X_j \to X_i$ 
is a homomorphism.
The covering solenoid associated with the sequence $X_1\leftarrow X_2\leftarrow\cdots$ is a locally compact abelian 
group \cite[Theorem 5]{Stone79}
and in this case the measure of the solenoid is 
a Haar measure.
For example, let $\ell(n)$ denote the least common multiple of
$1,\ldots,n$ and consider 
the case where $X_{\infty}$ is the inverse 
limit of all $X_n=\mathbb{R}/(\ell(n)\cdot\mathbb{Z})$. The covering maps
are homomorphisms between compact abelian groups. The inverse limit $X_{\infty}$ is also a compact
abelian group. Choose a translation-invariant measure $\mu_1$
on $X_1$.
Then, by construction, $\mu_{\infty}$ is translation-invariant. 
$\mu_{\infty}$
is a Haar measure for $X_{\infty}$ in this
case.
\end{eg}

For every $p\geq 1$, let $L^p(X_{\infty})$ denote 
the space of functions $f: X_{\infty} \to \Cbb$
such that $ \int_{X_{\infty}} |f|^p d \mu_{\infty} < \infty$. 
Let $L^p(X_i)$ denote the space of functions $h: X_{i} \to \Cbb$ such that 
$ \int_{X_{i}} |h|^p d \mu_i < \infty$.

\begin{lem}
For each $i$ and for each $h: X_i \to \Cbb$ that is integrable with respect to $\mu_i$, the function $P_{\infty, i}^*(h)$ is integrable with respect to $\mu_{\infty}$
and moreover 
\[  
\int_{X_{\infty}} P_{\infty, i}^*(h)\, d\mu_{\infty}~
=~
\int_{X_i} h\, d \mu_{i}. \]
\end{lem}
\begin{proof}
Every $\mu_{\infty}$-integrable function $f$ on $X_{\infty}$ 
satisfies the formula 
\[ \int_{X_{\infty}}f d\mu_{\infty}=\int_{X_{i}}\int_{P_{\infty,i}^{-1}(x_i)}f d\nu_{x_{i}} d\mu_i(x_{i})
\] 
see \cite[Formula (6)]{Goetz59}. 
For $f=P_{\infty,i}^*(h)$, one can check by using simple functions that $f$ is integrable. Then 
\begin{align*}
&\quad\int P_{\infty,i}^*(h)d\mu_{\infty}
=\int_{X_{i}}\int_{P_{\infty,i}^{-1}(x_i)}P_{\infty,i}^*(h) d\nu_{x_{i}} d\mu_i(x_{i})\\
&=\int_{X_{i}}h(x_i) \cdot \nu_{x_i}(P_{\infty,i}^{-1}(x_i)) d\mu_i(x_{i})
=\int_{X_{i}}h d\mu_i.
\end{align*}

\end{proof}

\begin{coro}\label{pullback of L^2 is L^2}
For every $p\geq 1$ and every $i\in\mathbb{N}$, if $h \in L^p(X_i)$, then $P_{\infty,i}^*(h) \in L^p(X_{\infty})$ and $$||P_{\infty,i}^*(h)||_{L^p(X_{\infty})}~ =~ ||h||_{L^p(X_i)}. $$
And for each $i\in\mathbb{N}$, we have $P_{\infty,i}^*(L^p(X_i))\subset L^p(X_{\infty})$.
\end{coro}

Let $\Pcal$ be the collection of sets of the form $P_{\infty,i}^{-1}(E)$ where $i\in\Nbb$
and $E$ is a Borel subset of $X_i$ such that $\mu_i(E)<\infty$.

\begin{lem} \label{lem:sigma-gen}
The Baire $\sigma$-algebra of $X_{\infty}$ is generated by $\Pcal$.
\end{lem}

\begin{proof}
Since $X_{\infty}$ is locally compact, Hausdorff, and second countable, 
the
Baire $\sigma$-algebra equals the Borel $\sigma$-algebra on $X_{\infty}$.
The Borel $\sigma$-algebra of $X_{\infty}$ is the subspace $\sigma$-algebra of the Borel $\sigma$-algebra of $\prod X_k$. 

The Borel $\sigma$-algebra of a countable product 
of second countable topological spaces is the product of the Borel
$\sigma$-algebras. So the Borel $\sigma$-algebra of $\prod X_i$ is generated by
$\{\pi_{i}^{-1}(E): i\in\Nbb, E\subset X_i \,Borel\, measurable\}$,
where 
$\pi_i$ is the projection from $\prod X_k$ to $X_i$. 
Since $\pi_i^{-1}(E)\cap X_{\infty}=P_{\infty,i}^{-1}(E)$, the 
subspace Borel 
$\sigma$-algebra of $X_{\infty}$ is generated by 
$\{P_{\infty,i}^{-1}(E): i\in\Nbb, E\subset X_i \,Borel\, measurable\}$. 

For each $i\in\Nbb\cup \{\infty\}$, 
$X_i$ is $\sigma$-finite. So the subspace Borel 
$\sigma$-algebra of $X_{\infty}$ is generated by $\Pcal$.  Equivalently, the Baire $\sigma$-algebra of $X_{\infty}$ is generated by $\Pcal$.
\end{proof}

Let ${\mathbbm 1}_Z$ denote the characteristic function for each set $Z$. 

\begin{prop}\label{approximation in L^2}
For every $p\geq 1$, every $L^p(X_{\infty})$ integrable function can be approximated in $L^p$ norm by sequences of finite linear combinations of functions in the set $\bigcup_{i\in\mathbb{N}} P_{\infty,i}^*(L^p(X_i))$.
\end{prop}

\begin{proof}
For each function $f\in L^p(X_{\infty})$,
using Lemma \ref{lem:sigma-gen}, one can show
that $f$ can be approximated, in $L^p$ norm, by finite linear combinations of characteristic functions associated with the sets in $\Pcal$. 
We have ${\mathbbm 1}_{P_{\infty,i}^{-1}(E)}=P_{\infty,i}^*({\mathbbm 1}_{E})$
and so $f$ can be approximated by finite linear combinations of pullbacks of characteristic functions. 
Each such characteristic function 
is an $L^p$ function, by definition of $\mathcal{P}$. The conclusion follows. 
\end{proof}

For each $i\in\mathbb{N}\cup\{\infty\}$, on $L^2(X_i)$ there exists a natural inner product $\langle\cdot,\cdot\rangle_{X_i}:L^2(X_i)\times L^2(X_i)\to \Cbb$ defined as follows: 
\[\langle \alpha,\beta\rangle=\int \alpha\cdot \bar{\beta}d\mu_i\] 
for every $\alpha,\beta\in L^2(X_i)$. Corollary 2.7 gives rise to:

\begin{coro}
For each $i\in\mathbb{N}$, and each $\alpha,\beta\in L^2(X_i)$, 
\[\langle \alpha,\beta\rangle_{X_i}
=\langle P_{\infty,i}^*(\alpha),P_{\infty,i}^*(\beta)\rangle_{X_{\infty}}.\]
\end{coro}

\section{Laplacian}

For the rest of the paper, we will assume that each $X_i$
is a manifold.

Each path connected component of $X_{\infty}$ is called a 
{\em leaf}. In this section, we define a leaf-wise Laplacian 
on the covering solenoid associated to $X_1 \leftarrow X_2 \leftarrow 
\cdots$.

We first apply results from \cite{McCord65}. Each leaf is dense 
in $X_{\infty}$, all leaves of $X_{\infty}$ are homeomorphic to 
each other and the fundamental group of each leaf is $\bigcap_i G_i$,
by section 5 of \cite{McCord65}. In particular, each leaf is a cover of $X_i$, 
for each integer $i$. The fact that each $X_i$ is a manifold implies
that each leaf is a manifold.

Take a complete Riemannian metric $g$ on $X_1$. Pull back the metric to $X_i$ by $P_{i,1}$ for every integer $i$. The finite covering map $P_{j,k}:X_j\to X_k$ is a local isometry for every $j,k\in\mathbb{N}$, $j>k$.

The map $P_{\infty,1}:X_{\infty}\to X_1$ restricted to a fixed leaf $\ell$ is a covering map from $\ell$ to $X_1$.
Pull back the complete Riemannian metric $g$ on $X_1$ and get the metric $(P_{\infty,1}|_{\ell})^*(g)$ on the leaf $\ell$. Then the covering map $P_{\infty,1}|_{\ell}$ becomes a local isometry. Since the Riemannian metric on each $X_i$ is the pullback metric for each integer $i$, the covering map $P_{\infty,i}|_{\ell}$ is also a local isometry.

\begin{defn}
Let $\mathcal{S}$ denote the space of functions 
$u:X_{\infty}\to\Cbb$ such that the the restriction of 
$u$ to each leaf is smooth.
If $u \in \mathcal{S}$ and $x$ belongs to the leaf 
$\ell$, define 
\[ (\Delta_{X_{\infty}}u)(x)~ =~ \Delta_{\ell} u|_{\ell} (x).
\]
where $\Delta_{\ell}$ is the Laplacian defined on $C^{\infty}(\ell)$.
\end{defn}

\begin{prop}
For each integer $i$ and each function $h\in C_c^{\infty}(X_i)$, $P_{\infty,i}^*(h)$ is smooth leaf-wise and 
$\Delta_{X_{\infty}} P_{\infty,i}^*(h)=P_{\infty,i}^*(\Delta_{X_i}h)$.
\end{prop}

\begin{proof}
Since the map $P_{\infty,i}:X_{\infty}\to X_i$ restricted to each leaf $\ell$
is a covering map from $\ell$ to $X_1$, the pullback of smooth function $h$ will
be a smooth function $(P_{\infty,i}|_{\ell})^*(h)$ on $\ell$. Thus
$P_{\infty,i}^*(h)$ is smooth leaf-wise and lies in $\mathcal{S}$. 

Since $P_{\infty,i}|_{\ell}$ is a local isometry, $P_{\infty,i}|_{\ell}$
satisfies that on each small open neighborhood, $(P_{\infty,i}|_{\ell})^*$ commutes with the Laplacian. To be precise, 
\[\Delta_{(P_{\infty,1}|_{\ell})^*(g)}(P_{\infty,i}^*(h)|_{\ell})=P_{\infty,i}^*
(\Delta_{X_i}h)|_{\ell}.\]

Since the choice of leaf $\ell$ is arbitrary and
$\Delta_{X_{\infty}}$ is defined to be taking Laplacian leaf-wise,
\[\Delta_{X_{\infty}} P_{\infty,i}^*(h)=P_{\infty,i}^*(\Delta_{X_i}h).\]
\end{proof}

\begin{coro}
For each integer $i$ and each eigenfunction $h$ of $\Delta_{X_i}$,
$P_{\infty,i}^*(h)$ is an eigenfunction of $\Delta_{X_{\infty}}$ and is of the same eigenvalue as $h$.
\end{coro}

\begin{proof}
Say $h$ is of eigenvalue $\lambda$. Notice that $h\in L^2(X_i)$ is smooth on $X_i$ as an eigenfunction of $\Delta_{X_i}$. Proposition 3.2 implies that
\[\Delta_{X_{\infty}} P_{\infty,i}^*(h)=P_{\infty,i}^*(\Delta_{X_i}h)=P_{\infty,i}^*(\lambda h)=\lambda P_{\infty,i}^*(h).\]
\end{proof}

Notice that Proposition 3.2 and Corollary 3.3 only concerned the 
Laplacian acting on smooth functions. However, we are interested 
in the self-adjoint Laplacians since the spectral theorem holds for self-adjoint operators.

In order to construct a self-adjoint Laplacian, 
we first restrict 
the domain of the Laplacian to a dense subset of $L^2(X_{\infty})$
on which the operator is symmetric. We will then show that
this symmetric operator has a self-adjoint extension.

Note that Corollary 2.7 and Proposition 3.2 imply that 
$\bigcup_{i\in\mathbb{N}} P_{\infty,i}^*(C_c^{\infty}(X_i))$ 
is a subset of $\mathcal{S}\cap L^2(X_{\infty})$ 
and $\Delta_{X_{\infty}}(\bigcup_{i\in\mathbb{N}}
P_{\infty,i}^*(C_c^{\infty}(X_i)))\subset L^2(X_{\infty})$.

\begin{prop}
$\bigcup_{i\in\mathbb{N}}
P_{\infty,i}^*(C_c^{\infty}(X_i))$ is a dense subset of  $L^2(X_{\infty})$ with respect to the $L^2(X_{\infty})$ norm.
\end{prop}

\begin{proof}
Proposition 2.9 implies that the collection of finite linear combinations of
functions in the set $\bigcup_{i\in\mathbb{N}} P_{\infty,i}^*(L^2(X_i))$ is a
dense subset of $L^2(X_{\infty})$ with respect to the $L^2(X_{\infty})$ norm.

Notice that $C_c^{\infty}(X_i)$ is dense in $L^2(X_i)$ with respect to the
$L^2$ norm for each $i$. Corollary 2.7 implies that
$P_{\infty,i}^*(C_c^{\infty}(X_i))$ is dense in $P_{\infty,i}^*(L^2(X_i))$ with
respect to the $L^2(X_{\infty})$ norm for each $i$. The conclusion follows.
\end{proof}

Therefore by restricting the domain of $\Delta_{X_{\infty}}$ to $\bigcup_{i\in\mathbb{N}}
P_{\infty,i}^*(C_c^{\infty}(X_i))$, we have a Laplacian defined on a dense subset of $L^2(X_{\infty})$ with image in $L^2(X_{\infty})$. 

We now show that $\Delta_{X_{\infty}}$ restricted to $\bigcup_{i\in\mathbb{N}}
P_{\infty,i}^*(C_c^{\infty}(X_i))$ 
is a symmetric nonnegative operator.

\begin{prop}
The operator $\Delta_{X_{\infty}}: \bigcup_{i\in\mathbb{N}}
P_{\infty,i}^*(C_c^{\infty}(X_i))\to
L^2(X_{\infty})$ is symmetric and nonnegative.
\end{prop}

\begin{proof}
Given any open cover $\{U_{\ell}\}$ of $X_1$, there exists a
partition of unity $\{\rho_{\ell}\}$ on $X_1$ such that the support of $\rho_{\ell}$ is
a compact subset of $U_{\ell}$ for each $\ell$. Choose the open cover
$\{U_{\ell}\}$ such that there exists a local trivialization on each
$U_{\ell}$. So there exists a partition of unity $\{\zeta_{\ell}\}$ on
$X_{\infty}$ such that the support of $\zeta_{\ell}$ is a compact subset of
$\pi^{-1}(U_{\ell}) \cong G_{\infty}\times U_{\ell}$ for each $\ell$.

As a result, it suffices to consider each 
$f,u$ that lie in $\bigcup_{i\in\mathbb{N}}
P_{\infty,i}^*(C_c^{\infty}(X_i))$ with support as 
a subset of $G_{\infty}\times U$, where $U\subset X_1$ is open.  

Let $g$ denote the Riemannian
metric on each leaf. We apply the fact that the measure on
$X_{\infty}$ is the product measure as in Definition 2.1:
\begin{align*}
&\quad\int_{X_{\infty}}\Delta_{X_{\infty}} f\cdot \bar{u}
=\int_{G_{\infty}\times U}\Delta_{X_{\infty}} f\cdot \bar{u}
=\int_{G_{\infty}}\int_{U}\Delta_{X_{\infty}} f\cdot \bar{u}\\
&=\int_{G_{\infty}}\int_{U} g(\nabla f\,, \nabla u)
=\int_{G_{\infty}}\int_{U} f\cdot \overline{\Delta_{X_{\infty}} u}
=\int_{X_{\infty}} f\cdot \overline{\Delta_{X_{\infty}} u}.
\end{align*}

The reason that there is no boundary term in the fourth 
integral is that the support of $f$ and $u$ restricted on each leaf is homeomorphic to a subset of $U$.
So the operator is symmetric.

For the nonnegativeness of $\Delta_{X_{\infty}}$, it again suffices to consider each function $f$ in $\bigcup_{i\in\mathbb{N}}
P_{\infty,i}^*(C_c^{\infty}(X_i))$ with support as a subset of $G_{\infty}\times U$. 
\begin{align*}
\int_{X_{\infty}}\Delta_{X_{\infty}} f\cdot \bar{f}
&=\int_{G_{\infty}\times U}\Delta_{X_{\infty}} f\cdot \bar{f}
=\int_{G_{\infty}}\int_{U}\Delta_{X_{\infty}} f\cdot \bar{f}\\
&=\int_{G_{\infty}}\int_{U}g(\nabla f\,, \nabla f)\geq 0.
\end{align*}

Again there is no boundary term in the fourth integral. So the operator is nonnegative.
\end{proof}

\begin{ass}
For the rest of this section, we will retrict the domain of $\Delta_{X_{\infty}}$ to $\bigcup_{i\in\mathbb{N}}
P_{\infty,i}^*(C_c^{\infty}(X_i))$.
\end{ass}

\begin{remark}
Since there is an inclusion map $P_{i+1,i}^*(C_c^{\infty}(X_i))\to C_c^{\infty}(X_{i+1})$ for each integer $i$, $C_c^{\infty}(X_1)\to C_c^{\infty}(X_2)\to C_c^{\infty}(X_3)\to \cdots$ is a direct sequence. 
Notice that $\bigcup_{i\in\mathbb{N}}
P_{\infty,i}^*(C_c^{\infty}(X_i))$ is
the direct limit of this sequence. Therefore it is reasonable to consider $\bigcup_{i\in\mathbb{N}}
P_{\infty,i}^*(C_c^{\infty}(X_i))$ as a domain of $\Delta_{X_{\infty}}$.
\end{remark}

We will now show
that the Laplacian on $X_{\infty}$
has a unique self-adjoint extension. 
The proof follows
the method that \cite{Strichartz83} developed for the Laplacian 
on a complete Riemannian manifold.

\begin{prop}
If $X_1$ is complete, then 
the operator $\Delta_{X_{\infty}}$ is essentially
self-adjoint.
\end{prop}

\begin{proof}
The operator $\Delta_{X_{\infty}}$ is essentially self-adjoint if and only if there are no
eigenfunctions of positive eigenvalue in the domain of $\Delta_{X_{\infty}}^*$ \cite[P.136-137]{Reed-Simon75} or \cite[Lemma 2.1]{Strichartz83}.

Consider each function $u\in Dom(\Delta_{X_{\infty}}^*)$ such that
$\Delta_{X_{\infty}}^*u=\lambda u$ for some $\lambda>0$.
It suffices to show that $u=0$.

The assumptions $u\in Dom(\Delta_{X_{\infty}}^*)$ and
$\Delta_{X_{\infty}}^*u=\lambda u$ imply that $\langle
u,\Delta_{X_{\infty}}v\rangle=\langle \lambda u,v\rangle$ for every $v\in
\bigcup_i P_{\infty,i}^*(C_c^{\infty}(X_i))$. In particular,
\[\langle u, \Delta_{X_{\infty}}P_{\infty,i}^*(\beta)\rangle
=\langle \lambda u,P_{\infty,i}^*(\beta)\rangle.
\]
for each integer $i$ and each $\beta\in C_c^{\infty}(X_i)$.

Since $L^2(X_{\infty})=P_{\infty,i}^*(L^2(X_i))\bigoplus P_{\infty,i}^*(L^2(X_i))^{\perp}$, $u=P_{\infty,i}^*(\alpha)+h$ where $\alpha\in L^2(X_i)$ and $h\in
P_{\infty,i}^*(L^2(X_i))^{\perp}$. Then
\[
\langle P_{\infty,i}^*(\alpha),\Delta_{X_{\infty}}P_{\infty,i}^*(\beta)\rangle+
0
=\langle \lambda P_{\infty,i}^*(\alpha),P_{\infty,i}^*(\beta)\rangle+0.
\]

Proposition 3.2 implies that
\[
\langle P_{\infty,i}^*(\alpha),P_{\infty,i}^*(\Delta_{X_{i}}\beta)\rangle
=\langle \lambda P_{\infty,i}^*(\alpha),P_{\infty,i}^*(\beta)\rangle.
\]

Corollary 2.10 implies that the above equation is equivalent with
\[\langle \alpha,\Delta_{X_{i}}\beta\rangle
=\langle \lambda \alpha,\beta\rangle.\]

Therefore $\alpha\in Dom(\Delta_{X_{i}}^*)$ and  
$\Delta_{X_{i}}^*\alpha=\lambda \alpha$ where $\lambda>0$. This implies that
$\alpha=0$ on complete manifold $X_i$ for each integer $i$ \cite[Theorem 3]{Yau76}. 
So the projection of $u$ onto $L^2(X_i)$ equals 0 for each integer $i$. 

For each integer $i$, the space $P_{i+1,i}^*(L^2(X_i))$ is a closed subspace of $L^2(X_{i+1})$. Let $V_{i+1}$ denote $P_{i+1,i}^*(L^2(X_i))^{\perp}$. Let $V_{1}$ denote $L^2(X_1)$. Proposition 2.9 implies that  $L^2(X_{\infty})=\overline{\bigoplus P_{\infty,i}^*V_{i}}$. 

The fact that the projection of $u$ onto $L^2(X_i)$ equals 0 for each integer $i$ implies that the projection of $u$ onto $V_j$ equals 0 for each $1\leq j\leq i$. So the projection of $u$ onto $V_j$ equals 0 for each $j\in\Nbb$. 

As a result, $u=0$. 
\end{proof}

\begin{remark}[Compact Manifolds with boundary]
Consider the case where each $X_i$ is a compact manifold with boundary. In order for each Laplacian to be symmetric, we consider the collection $L_i$ of $C^{\infty}(X_i)$ functions that satisfy
any Robin condition, for instance, the Dirichlet condition or the Neumann condition. Notice that $L_i$ is a dense subset of $C^{\infty}(X_i)$ under the $L^2$ norm. Then we consider the leafwise Laplacian on $\bigcup_{i\in\Nbb}P_{\infty,i}^*(L_i)$ for $X_{\infty}$. 
Notice that $\bigcup_i P_{\infty,i}^*(L_i)$ is a dense subset of $L^2(X_{\infty})$.

The results in this paper are still true for the case of a sequence of 
compact manifolds with boundary. 
\end{remark}

\begin{lem}
Let $A: \Hcal \to \Hcal$ and $A': \Hcal' \to \Hcal'$  
be closable, unbounded operators with
respective dense domains $\Ccal$ and $\Ccal'$. 
Suppose that $\phi: \Hcal \to \Hcal'$ is 
continuous, maps $\Ccal$ into $\Ccal'$, and that for 
each $u \in \Ccal$ we have 
\begin{equation*} \label{eqn:commute}
A' \circ \phi(u) = \phi \circ A(u).
\end{equation*}
Then $\phi$ maps the domain $\Dcal$ of the closure $\overline{A}$
into the domain $\Dcal'$ of the closure  $\overline{A'}$, and 
for each $u \in \Dcal$ we have 
\begin{equation*} \label{eqn:commute2}
\overline{A'} \circ \phi(u) = \phi \circ \overline{A}(u).
\end{equation*}
\end{lem}

\begin{proof}
The map $\phi$ naturally determines a map 
$\phi \times \phi: \Hcal \times \Hcal \to \Hcal' \times \Hcal'$.
We have 
\[ 
\phi \times \phi(\{(u, Au)~ |~ u \in \Ccal\})~
=~
\{(\phi(u), A \phi(u))~ |~ u \in \Ccal\}~
=~
\{(u, Au)~ |~ u \in \phi(\Ccal)\}.
\]

In particular, $\phi \times \phi$ maps the
graph of $A$ into the graph of $A'$. 
Since $\phi \times \phi$ is continuous,
the closure is mapped into the closure.
The claim follows. 
\end{proof}

The operator $\Delta_{X_i}$ is essentially self-adjoint for every $i\in\Nbb\cup\{\infty\}$. The self-adjoint extension of $\Delta_{X_i}$ is the closure of $\Delta_{X_i}$. Let $\overline{\Delta_{X_i}}$ denote the closure.

Consider the map $P_{\infty,i}^*:L^2(X_i)\to L^2(X_{\infty})$. $P_{\infty,i}^*$ preserves $L^2$ norm by Corollary 2.7, so $P_{\infty,i}^*$ is continuous. $P_{\infty,i}$ maps $C_c^{\infty}(X_i)$ to $\bigcup_i P_{\infty,i}^*(C_c^{\infty}(X_i))$. And for each $\alpha\in C_c^{\infty}(X_i)$,  we have $\Delta_{X_{\infty}} P_{\infty,i}^*(\alpha)=P_{\infty,i}^*(\Delta_{X_i}\alpha)$. Therefore Lemma 3.10 gives rise to the following:

\begin{coro}
For each integer $i$, we have $P_{\infty,i}^*(Dom(\overline{\Delta_{X_i}}))\subset Dom(\overline{\Delta_{X_{\infty}}})$.
\end{coro}

\begin{remark}
Notice that if $X_1$ is complete but not compact, eigenfunctions of the Laplacian operator of $X_i$, $i\in\Nbb$, may not be compactly supported. So they may not lie in $C_c^{\infty}(X_i)$, $i\in\Nbb$. Therefore the pullback functions of eigenfunctions onto $X_{\infty}$ may not necessarily lie in $\bigcup_i P_{\infty,i}^*C_c^{\infty}(X_i)$. However, eigenfunctions lie in the domain of the Laplacian. So the pullbacks of these eigenfunctions still lie in the domain of the Laplacian of $X_{\infty}$ by Corollary 3.11.
\end{remark}

\section{Spectrum}

If $X_1$ is a complete Riemannian manifold, the operator $\Delta_{X_{\infty}}$ is essentially self-adjoint by Proposition 3.8. The operator $\Delta_{X_{i}}$ is also essentially self-adjoint for each $i$. For the rest of the paper, we will let $\Delta_{X_i}$ denote the unique self-adjoint extension for every $i\in\Nbb\cup\{\infty\}$.

In this section, we discuss the spectrum of $\Delta_{X_{\infty}}$.

Let's first discuss the case where $X_1$---and hence each
$X_k$---is a closed manifold. 

For basic facts concerning the spectral theory of closed manifolds
we refer the reader to \cite{Rosenberg97}. Let $A_k$ be an orthonormal collection of eigenfunctions
of $\Delta_{X_k}$ whose algebraic span is dense in $L^2(X_k)$. Since $P_{k+1,k}^*(L^2(X_k))$ is a closed subspace of $L^2(X_{k+1})$, we can choose such collections so that $P_{k+1,k}^*(A_{k})\subset A_{k+1}$ for every integer $k$.

Since $P_{k+1,k}\circ P_{\infty,k+1}=P_{\infty,k}$, we have 
\[
P_{\infty,k}^{*}(A_{k})
=
(P_{k+1,k}\circ P_{\infty,k+1})^{*}(A_{k})
=
P_{\infty,k+1}^{*}( P_{k+1,k}^{*}(A_{k}))\subset P_{\infty,k+1}^{*}(A_{k+1})
\]
for every integer $k$. Then $\bigcup_k P_{\infty,k}^{*}(A_k)$ is a union of an increasing sequence. Corollary 2.7 implies that $\bigcup_k P_{\infty,k}^{*}(A_k)\subset L^2(X_{\infty})$. Corollary 2.10 implies that $\bigcup_k P_{\infty,k}^{*}(A_k)$ is an orthonormal collection. 

Let $\mathcal{A}$ denote the set 
\[
\left\{f\in L^2(X_{\infty})\, 
\left|~
\,f
=
\sum_{j=1}^{L}a_j e_j,\,
a_j\in\Cbb,\,
e_j\in\bigcup_k P_{\infty,k}^{*}(A_k),\,
L\in\Nbb
\right.
\right\}.
\]

Let $\overline{\mathcal{A}}$ denote the closure of $\mathcal{A}$ with respect to the $L^2(X_{\infty})$ norm.

\begin{prop}\label{prop:L2-solenoid}
$\overline{\mathcal{A}}=L^2(X_{\infty})$.
\end{prop}

\begin{proof}
Proposition 2.9 implies that every function $f\in L^2(X_{\infty})$ 
can be approximated by finite linear combinations of 
pullbacks of characteristic functions on $X_i$, $i\in\Nbb$. 
On each $X_i$, every characteristic function can be approximated 
by a finite linear combinations of elements in $A_i$ 
i.e. eigenfunctions of $\Delta_{X_i}$. 

Corollary 2.7 implies that the norm of each $L^2(X_i)$ function
equals the $L^2$ norm of its pullback in $L^2(X_{\infty})$. 
In particular, the pullback of each characteristic function associated 
to a subset of $X_i$ can be approximated by a finite linear combination 
of pullbacks of eigenfunctions of $\Delta_{X_i}$, $i\in\Nbb$, 
i.e. elements in $\mathcal{A}$. Therefore $f$ 
can be approximated by finite linear combinations of elements 
in $\mathcal{A}$.
\end{proof}

Let $\mathcal{M}$ denote the Lebesgue
$\sigma$-algebra on $\Rbb$.

Recall that, by the spectral theorem, each self-adjoint operator $T$
has a resolution of the identity. That is, there
exists a unique projection valued measure $E$ so that
$\langle T f,u\rangle=\int_{\Rbb}t\, d E_{f,u}(t)$
where $\omega \mapsto E_{f,u}(\omega) := \langle E(\omega)f,u\rangle$
is the complex-valued measure on $\Rbb$ 
associated to each $u$ and $f$ in $Dom(T)$ \cite{Rudin91}.

\begin{thm}
The resolution of the identity $E$ for $\Delta_{X_{\infty}}$ is 
given by the following formula
\[ 
\omega~
\mapsto~
\left(
f\mapsto\sum_{e\in\mathcal{A},\, \lambda(e) \in \omega}\langle f, e\rangle \cdot e 
\right)
\]
where $\omega\in\mathcal{M}$ and $\lambda(e)$ is the 
eigenvalue of $e$.
\end{thm}

\begin{proof}
First, we show that $E$ described below is a resolution of the identity:

Let $B(L^2(X_{\infty}))$ denote the collection of all the bounded 
linear operators of $L^2(X_{\infty})$ to $L^2(X_{\infty})$.
Define $E:\mathcal{M}\to B(L^2(X_{\infty}))$ by 
\[ 
\omega~
\mapsto~
\left(
f\mapsto\sum_{e\in\mathcal{A},\, \lambda(e) \in \omega}\langle f, e\rangle \cdot e 
\right)
\]
for each Lebesgue measurable set $\omega \in\mathcal{M}$. 

Now we check the conditions for $E$ to be a resolution of the identity. 
By definition, $E(\emptyset)=(f\mapsto \sum_{\lambda(e)\in\emptyset}\langle f,
e\rangle \cdot e)=0$. Proposition 4.1 implies that for each $f\in L^2(X_{\infty})$,
$f=\sum_{e\in\mathcal{A}} \langle f, e\rangle \cdot e$. So $E(\Rbb)=(f\mapsto \sum_{\lambda(e)\in\Rbb}\langle f, e\rangle \cdot e)=id$. 

For each $f\in L^2(X_{\infty})$ and each $\omega',\omega''\in\mathcal{M}$,

\begin{eqnarray*}
E(\omega')\circ E(\omega'')(f)&=&\sum_{\lambda(e')\in\omega'}
\left\langle \sum_{\lambda(e'')\in \omega''}\langle f,e''\rangle\cdot e'',\, e' 
\right\rangle\cdot e'\\
&=&\sum_{\lambda(e')\in\omega'}\sum_{\lambda(e'')\in \omega''}\langle f,e''\rangle\cdot \langle e'',e'\rangle\cdot e'\\
&=&\sum_{\lambda(e')\in\omega'\cap\omega''}\langle f,e'\rangle\cdot e'\\
&=&E(\omega'\cap\omega'')(f).
\end{eqnarray*}

Notice that the $L^2$ norm of $f$ is finite, so the series above is absolutely
convergent. Therefore the order of the sum can be changed in the above formula.
So $E(\omega')\circ E(\omega'')=E(\omega'\cap\omega'')$.

For each $\omega',\omega''\in\mathcal{M}$,
if $\omega'\cap\omega''=\emptyset$, 
by definition, $E(\omega'\cup\omega'')=E(\omega')+E(\omega'')$.

For each $f,u\in Dom(\Delta_{X_{\infty}})$, 
define $E_{f,u}: \Mcal \to \Cbb$ by 
$E_{f, u}(\omega)=\langle E(\omega)f,u \rangle$. 
Then $E_{f,u}(\omega)=\langle \sum_{\lambda(e)\in\omega}\langle f,e\rangle \cdot e, u\rangle$ 
for every $\omega\in\mathcal{M}$. $E_{f,u}$ satisfies that $E_{f,u}(\emptyset)=0$.
And for each sequence $\{\omega_s\}_{s=1}^{\infty}$ of disjoint sets in
$\mathcal{M}$, we have 
\begin{align*}
E_{f,u}(\bigcup_s\omega_s)
&=\left\langle
\sum_{\lambda(e)\in\bigcup_s\omega_s}\langle f,e\rangle \cdot e,\,
u \right\rangle
=
\left\langle 
\sum_s\sum_{\lambda(e)\in\omega_s}\langle f,\, e\rangle \cdot e,\,
u\right\rangle\nonumber\\
&=\sum_s\left\langle 
\sum_{\lambda(e)\in\omega_s}\langle f,\, e\rangle \cdot e,\,
u\right\rangle
=
\sum_s E_{f,u}(\omega_s).
\end{align*}

Since the $L^2$ norm of $f$ and $u$ is finite, the series above is absolutely convergent. Therefore the second equality in the above formula is true. We see that $E_{f,u}$ satisfies the above two conditions of a measure. Therefore $E_{f,u}$ is a measure for every
$f,u\in L^2(X_{\infty})$.

So $E$ is a resolution of the identity.

Second, we show that $E$ defined above is the resolution of the identity that corresponds to $\Delta_{X_{\infty}}$.

For each $f,u\in Dom(\Delta_{X_{\infty}})$, 
\begin{align*}
\int_{-\infty}^{\infty} t\, d E_{f,u}(t)
&=\sum_{e\in\mathcal{A}} \lambda(e)\cdot\langle\langle f, e\rangle\cdot e, u\rangle
=\langle \sum_{e\in\mathcal{A}} \lambda(e)\cdot\langle f, e\rangle\cdot e, u\rangle\\
&=\langle \sum_{e\in\mathcal{A}} \langle f, e\rangle\cdot \Delta_{X_{\infty}}e, u\rangle
=\langle \Delta_{X_{\infty}}f,u\rangle.
\end{align*}

Then the uniqueness implies that $E$ is the resolution of the identity that corresponds to $\Delta_{X_{\infty}}$ \cite[Theorem 13.30]{Rudin91}.
\end{proof}

Notice that by definition $E$ is supported on $\overline{\bigcup_i \sigma(\Delta_{X_i})}$.

\begin{coro}
$\sigma(\Delta_{X_{\infty}})=\overline{\bigcup_i \sigma(\Delta_{X_i})}$.
\end{coro}

Now we consider the more general case where $X_1$ is complete.

For each integer $i$, let $E^i$ denote the resolution of the identity for $\Delta_{X_i}$. 

\begin{prop}
For each $\omega\in\mathcal{M}$ and each $i\in\Nbb$, 
\[P_{i+1,i}^*\circ E^i(\omega)=E^{i+1}(\omega)\circ P_{i+1,i}^*.\] 
\end{prop}

\begin{proof}
The definition of $\Delta_{X_i}$ and $\Delta_{X_{i+1}}$ implies that
$P_{i+1,i}^*\circ \Delta_{X_i}=\Delta_{X_{i+1}}\circ P_{i+1,i}^*$. So for each
$\alpha\in Dom(\Delta_{X_i})$ and $\beta\in L^2(X_i)$,
\[\langle P_{i+1,i}^*(\Delta_{X_i}\alpha),P_{i+1,i}^* (\beta)\rangle
=\langle \Delta_{X_{i+1}}( P_{i+1,i}^*(\alpha)),P_{i+1,i}^*
(\beta)\rangle.\]

By properties of $E^{i+1}$,
\[\langle \Delta_{X_{i+1}}(P_{i+1,i}^*(\alpha)),P_{i+1,i}^* (\beta)\rangle
=\int_{\Rbb} t\, d E^{i+1}_{P_{i+1,i}^*(\alpha),P_{i+1,i}^* (\beta)}(t).\]

Corollary 2.10 implies that $\langle P_{i+1,i}^*(\Delta_{X_i}\alpha),P_{i+1,i}^* (\beta)\rangle=\langle \Delta_{X_i}\alpha,\beta\rangle$. 

By properties of $E^i$,
\[\langle \Delta_{X_i}\alpha,\beta\rangle=\int_{\Rbb}t\, d E^i_{\alpha,\beta}(t).\]

The above formulas imply that \[\int_{\Rbb}t\, d E^i_{\alpha,\beta}(t)=\int_{\Rbb} t\, d E^{i+1}_{P_{i+1,i}^*(\alpha),P_{i+1,i}^* (\beta)}(t).\]

The above equation is true for not only $t$ but also functions that are essentially bounded \cite[Theorem 12.21]{Rudin91}. In particular, each compactly supported continuous function $\zeta$ on $\Rbb$ satisfies the equation:
\[\int_{\Rbb}\zeta(t)\, d E^i_{\alpha,\beta}(t)=\int_{\Rbb} \zeta(t)\, d E^{i+1}_{P_{i+1,i}^*(\alpha),P_{i+1,i}^* (\beta)}(t).\]

So for each $\omega\in \mathcal{M}$, $E^i_{\alpha,\beta}(\omega)=E^{i+1}_{P_{i+1,i}^*(\alpha),P_{i+1,i}^* (\beta)}(\omega)$. Then
\[ \langle E^i(\omega)\alpha,\beta\rangle=\langle E^{i+1}(\omega)P_{i+1,i}^*(\alpha),P_{i+1,i}^* (\beta)\rangle.\]

Again by Corollary 2.10,
\[\langle E^i(\omega)\alpha,\beta\rangle=\langle P_{i+1,i}^*E^{i}(\omega)(\alpha),P_{i+1,i}^* (\beta)\rangle.\]

So the following equality holds:
\[\langle P_{i+1,i}^*E^{i}(\omega)(\alpha),P_{i+1,i}^* (\beta)\rangle=
\langle E^{i+1}(\omega)P_{i+1,i}^*(\alpha),P_{i+1,i}^* (\beta)\rangle.\]

The conclusion follows.
\end{proof}

For each integer $i$, the space $P_{i+1,i}^*(L^2(X_i))$ is a closed subspace of $L^2(X_{i+1})$. Let $W_{i+1}$ denote the subspace $P_{i+1,i}^*(L^2(X_i))$. Let $V_{i+1}$ denote $P_{i+1,i}^*(L^2(X_i))^{\perp}$. Then $L^2(X_{i+1})=W_{i+1}\bigoplus V_{i+1}$. Let $V_1$ denote $L^2(X_1)$.

\begin{coro}
For each integer $i$ and each $\omega\in\mathcal{M}$, $E^{i+1}(\omega)$ maps $W_{i+1}$ to $W_{i+1}$ and maps $V_{i+1}$ to $V_{i+1}$.
\end{coro}

\begin{proof}
For each $h\in L^2(X_i)$, we have $P_{i+1,i}^*(h)\in W_{i+1}$. Proposition 4.4 implies that $E^{i+1}(\omega)(P_{i+1,i}^*(h))=P_{i+1,i}^*(E^i(\omega)h)$ which lies in $W_{i+1}$. So $E^{i+1}(\omega)$ maps $W_{i+1}$ to $W_{i+1}$. Since $V_{i+1}=W_{i+1}^{\perp}$, $E^{i+1}(\omega)$ maps $V_{i+1}$ to $V_{i+1}$.
\end{proof}

For each $i$, let $Q_i$ denote the composition of the projection map $L^2(X_{\infty})\to P_{\infty,i}^*(V_i)$ and the identification map $P_{\infty,i}^*(V_i)\to V_i$.

\begin{thm}
The resolution of the identity $E$ for $\Delta_{X_{\infty}}$ is 
given by the following formula 
\[ 
\omega~
\mapsto~
\left(
f\mapsto\sum_i P_{\infty,i}^*(E^i(\omega)|_{V_i}Q_i(f)
\right).
\]
\end{thm}

\begin{proof}
First, we show that $E$ described below is a resolution of the identity:

Define $E:\mathcal{M}\to B(L^2(X_{\infty}))$ by $\omega\mapsto(f\mapsto \sum_i P_{\infty,i}^*(E^i(\omega)|_{V_i}Q_i(f))$. Corollary 4.5 implies that image of $E^i(\omega)|_{V_i}$ lies in $V_i$. So the image of $P_{\infty,i}^*(E^i(\omega)|_{V_i})$ lies in $P_{\infty,i}^*(V_i)$. Since all $P_{\infty,i}^*(V_i)$ are subspaces of $L^2(X_{\infty})$ and are mutually orthogonal to each other, the sum makes sense and lies in $L^2(X_{\infty})$. 

Now we check all the conditions for $E$ to be a resolution of the identity. By definition, $E(\emptyset)=0$. $E(\Rbb)(f)$ is the sum of projections of $f$ onto $P_{\infty,i}^*(V_i)$. Proposition 2.9 implies that the sum of projections of $f$ onto all $P_{\infty,i}^*(V_i)$ is equal to $f$. So $E(\Rbb)=id$. 

For each $\omega'$, $\omega''\in\mathcal{M}$ and each $f\in L^2(X_{\infty})$, by definition of $E$:
\begin{align*}
\quad E(\omega')\circ E(\omega'')(f)
&=E(\omega')\big[\sum_j P_{\infty,j}^*(E^j(\omega'')|_{V_j}Q_j(f))\big] \\
&=\sum_j E(\omega')\big[P_{\infty,j}^*(E^j(\omega'')|_{V_j}Q_j(f))\big]\\
&=\sum_{j,i} P_{\infty,i}^*(E^i(\omega')|_{V_i}\big[Q_i(  P_{\infty,j}^*(E^j(\omega'')|_{V_j}Q_j(f)) )\big]). 
\end{align*}

We apply the orthogonality of all $P_{\infty,k}^*(V_k)$:
\begin{align*}
\quad E(\omega')\circ E(\omega'')(f) 
&=\sum_{j,i} P_{\infty,i}^*(E^i(\omega')|_{V_i}\big[Q_i(  P_{\infty,j}^*(E^j(\omega'')|_{V_j}Q_j(f)) )\big])\\
&=\sum_i P_{\infty,i}^*(E^i(\omega')|_{V_i}\big[Q_i( P_{\infty,i}^*(E^i(\omega'')|_{V_i}Q_i(f)) )\big]) \\
&=\sum_i P_{\infty,i}^*(E^i(\omega')|_{V_i}\big[  E^i(\omega'')|_{V_i}Q_i(f)\big]) \\
&=\sum_i P_{\infty,i}^*(E^i(\omega'\cap\omega'')|_{V_i}Q_i(f)) 
=E(\omega'\cap\omega'')(f).
\end{align*}

Notice by definition of $Q_i$, we have $Q_i\circ P_{\infty,i}^*=id$. So the third equality above is true. Therefore $E(\omega'\cap\omega'')=E(\omega')\circ E(\omega'')$.

For each $\omega',\omega''\in\mathcal{M}$, if $\omega'\cap\omega''=\emptyset$, the definition of $E$ implies that $E(\omega'\cup\omega'')=E(\omega')+E(\omega'')$.

For each $f,u\in L^2(X_{\infty})$, 
define $E_{f,u}: \Mcal \to \Cbb$ by 
$E_{f, u}(\omega)=\langle E(\omega)f,u \rangle$. $E_{f,u}$ satisfies that $E_{f,u}(\emptyset)=\langle E(\emptyset)f,u \rangle=0$.
And for each sequence $\{\omega_s\}_{s=1}^{\infty}$ of disjoint sets in
$\mathcal{M}$, we have 
\[ 
E_{f,u}(\bigcup_s\omega_s)
~=\left\langle E(\bigcup_s\omega_s)f\,,
u \right\rangle~
=~
\left\langle 
\sum_s E(\omega_s)f\,,\,
u\right\rangle~ 
=~
\sum_s E_{f,u}(\omega_s).
\]

The second equality is true because, for each summand of $E$, $E^i$ satisfies the equation $\langle E^i(\cup_s\omega_s)Q_i(f),Q_i(u)\rangle=\sum_s \langle E^i(\omega_s)Q_i(f),Q_i(u)\rangle$. Since the $L^2$ norm of $f$ is finite, the series above is absolutely convergent. Therefore the order of the sum in the above formula can be changed. 

As a result, $E_{f,u}$ satisfies the two conditions for a measure. Therefore $E_{f,u}$ is a measure for every
$f,u\in L^2(X_{\infty})$.

So $E$ defined above is a resolution of the identity.

Second, we show that the $E$ defined above is the resolution of the identity that corresponds to $\Delta_{X_{\infty}}$. It suffices to check if the equation $\int_{\Rbb}t \, d E_{f,u}(t)=\langle \Delta_{X_{\infty}}f,u\rangle$ holds for every $f,u\in Dom(\Delta_{X_{\infty}})$. 

Since the collection of finite linear combinations of pullbacks of $L^2(X_i)$ functions, $i\in\Nbb$, is a dense subset of $L^2(X_{\infty})$ and $C_c^{\infty}(X_i)$ is a dense subset of $L^2(X_i)$ for each $i$, it suffices to check for the case where $f=P_{\infty,k}^*(\alpha)$ and $u=P_{\infty,k}^*(\beta)$ for some integer $k$ and $\alpha,\beta\in C_c^{\infty}(X_k)$. 

Notice that for every $\omega\in\mathcal{M}$,

\begin{align*}
E(\omega)f&=\sum_i P_{\infty,i}^*(E^i(\omega)Q_i(f))\\
&=\sum_{i\leq k} P_{\infty,i}^*(E^i(\omega)Q_i(P_{\infty,k}^*(\alpha)))+0\\
&=\sum_{i\leq k} P_{\infty,k}^*P_{k,i}^*(E^i(\omega)Q_i(P_{\infty,k}^*(\alpha)))\\
&=\sum_{i\leq k} P_{\infty,k}^*(E^k(\omega)P_{k,i}^*Q_i(P_{\infty,k}^*(\alpha)))\\
&=P_{\infty,k}^*(E^k(\omega)\sum_{i\leq k}P_{k,i}^*Q_i(P_{\infty,k}^*(\alpha)))\\
&=P_{\infty,k}^*(E^k(\omega)\alpha).
\end{align*}

As a result, 
\[E_{f,u}(\omega)
=\langle E(\omega)f,u\rangle
=\langle P_{\infty,k}^*(E^k(\omega)\alpha),P_{\infty,k}^*(\beta)\rangle
=\langle E^k(\omega)\alpha,\beta\rangle
=E^k_{\alpha,\beta}(\omega).
\]

So 
\begin{align*}
&\int_{\Rbb}t \, d E_{f,u}(t)
=\int_{\Rbb}t \, d E^k_{\alpha,\beta}(t)
=\langle \Delta_{X_k}\alpha,\beta\rangle\\
&=\langle P_{\infty,k}^*\Delta_{X_k}\alpha,P_{\infty,k}^*\beta\rangle
=\langle \Delta_{X_{\infty}}P_{\infty,k}^*\alpha,P_{\infty,k}^*\beta\rangle
=\langle \Delta_{X_{\infty}}f,u\rangle.
\end{align*}

So $E$ is the resolution of the identity that corresponds to $\Delta_{X_{\infty}}$ from the uniqueness by \cite[Theorem 13.30]{Rudin91}. 
\end{proof}

Notice that by definition $E$ is supported on $\overline{\bigcup_i \sigma(\Delta_{X_i})}$, which implies the following:

\begin{thm}\label{mainTheorem}
If $X_1$ is a complete Riemannian manifold, then $\sigma(\Delta_{X_{\infty}})=\overline{\bigcup_i\sigma(\Delta_{X_{i}})}$.
\end{thm}

\section{An equivalent statement of Selberg's 1/4 conjecture}

Consider the modular group $SL(2,\Zbb)$ and its congruence subgroups. The congruence subgroup of level $n$ is
$\Gamma(n):=\{\gamma\in SL(2,\Zbb)\,|\,\gamma\equiv I\,(mod\,\,n)\}$. Let $\Hbb$
denote the upper half plane with the hyperbolic metric. Consider the modular
surfaces $X(n):=\Gamma(n)\backslash\Hbb$, $n\in\Nbb$. Each modular surface
$X(n)$ is a finite area, non-compact, hyperbolic surface. It is well known that
$X(n)$ is a complete hyperbolic surface for every $n>1$. See for example
\cite{Bergeron16}.

$\Gamma(k)$ is without elliptic elements for every $k>1$, \cite[Section
2.3]{Bergeron16}. So the fundamental group of $X(k)$ is without elliptic
elements for every $k>1$. Consider each $\Gamma(n)$ with $n>1$. Let $m$ be a
multiple of $n$. $\Gamma(m)$ is a finite index normal subgroup of $\Gamma(n)$.
$X(m)$ is an unbranched regular finite cover of $X(n)$. 

For each $k\in \Nbb$, let $\ell(k)\in \Nbb$ denote the least common multiple of
the integers from $2$ to $k$. Then $X(\ell(k))$ is a regular finite cover of
$X(2),\ldots,X(k)$. Therefore $\sigma(\Delta_{X_{\ell(k)}})$ contains
$\sigma(\Delta_{X_2}),\ldots,\sigma(\Delta_{X_k})$. 

The sequence $X(\ell(2))\leftarrow X(\ell(3))\leftarrow X(\ell(4))\leftarrow\cdots$ is a sequence of regular finite coverings of complete Riemannian manifolds. Let $X(\infty)$ denote the inverse limit of
$\{X(\ell(k))|k>1\}$.
Theorem 4.7 implies that $\sigma(\Delta_{X(\infty)})=\overline{\bigcup_{k>1}
\sigma(\Delta_{X_{\ell(k)}})}$. Notice that
\[\bigcup_{j>1}\sigma(\Delta_{X_j})=\bigcup_{k>1}\bigcup_{1<i\leq
k}\sigma(\Delta_{X_i})\subset \bigcup_{k>1}
\sigma(\Delta_{X_{\ell(k)}})\subset\bigcup_{j>1}\sigma(\Delta_{X_j}).\] 

So $\sigma(\Delta_{X(\infty)})$ is equal to $\overline{\bigcup_{k>1}
\sigma(\Delta_{X_{\ell(k)}})}=\overline{\bigcup_{j>1}\sigma(\Delta_{X_j})}$.

It is known that the first nonzero eigenvalue of $\Delta_{X_1}$ satisfies
$\lambda_1(\Gamma(1)\backslash\Hbb)>1/4$. See for example \cite[Theorem
3.38]{Bergeron16}. Thus the statement of Selberg's $1/4$ conjecture that the
first nonzero eigenvalue $\lambda_1(X(n))\geq 1/4$ for every integer $n$ is
equivalent
to the statement that $\lambda_1(X(n))\geq 1/4$ for every $n>1$. The latter can be turned into a statement about spectrum of $X(\infty)$. Theorem 4.7 implies the following:

\begin{coro}\label{corollary}
Selberg's 1/4 conjecture is true if and only if the spectrum of $\Delta_{X(\infty)}$ does not intersect $(0,1/4)$.
\end{coro}


\end{document}